\documentclass[11pt,twoside]{amsart} 

\usepackage{prefix}

\title{Extendability of $1$-decomposable complexes}

 \author{Rhea Ghosal}
 \address{Westlake High School} 
 \email{grhea1008@gmail.com}
 \author{Melody Han}
 \address{St. John's School} 
 \email{melodyh2727@gmail.com}
 \author{Benjamin Keller}
 \address{Princeton University}
 \email{bk9595@princeton.edu} 
 \author{Scarlett Kerr}
 \address{Byron Nelson High School} 
 \email{scarlettpk333@gmail.com}
 \author{Justin Liu}
 \address{Westwood High School} 
 \email{justinl08182009@gmail.com}
 \author{Suho Oh}
 \address{Texas State University} 
 \email{suhooh@txstate.edu}
 \author{Ryan Tang}
 \address{Lexington High School} 
 \email{tang.ryan8@gmail.com}
  \author{Chloe Weng}
 \address{Clements High School} 
 \email{chloe0909.weng@gmail.com}

\begin{document}
\maketitle

\begin{abstract}
A well-known conjecture of Simon (1994) states that any pure $d$-dimensional shellable complex on $n$ vertices can be extended to $\Delta_{n-1}^{(d)}$, the $d$-skeleton of the $(n-1)$-dimensional simplex, by attaching one facet at a time while maintaining shellability.

The notion of $k$-decomposability for simplicial complexes, which generalizes shellability, was introduced by Provan and Billera (1980). Coleman, Dochtermann, Geist, and Oh (2022) showed that any pure $d$-dimensional $0$-decomposable complex on $n$ vertices can similarly be extended to $\Delta_{n-1}^{(d)}$, attaching one facet at a time while preserving $0$-decomposability.

In this paper, we investigate the analogous question for $1$-decomposable complexes. We prove a slightly relaxed version: any pure $d$-dimensional $1$-decomposable complex on $n$ vertices can be extended to $\Delta_{n + d - 3}^{(d)}$, attaching one facet at a time, while maintaining $1$-decomposability.
\end{abstract}

\section{Introduction}

A pure simplicial complex $\mC$ is termed \emph{shellable} if one can order its facets $F_1, F_2, \dots, F_s$ so that each $F_i$ intersects the preceding facets in a pure complex of codimension one. Shellability serves as a vital combinatorial method with implications for both the topology of $\mC$ and the algebraic properties of its Stanley-Reisner ring. Notable instances of shellable simplicial complexes include independence complexes of matroids \cite{pb80}, boundary complexes of simplicial polytopes \cite{bm72}, and the skeleta of shellable complexes \cite{bw96}. Specifically, for any $k = 1, 2, \dots, n-1$, the $k$-skeleton of a simplex with vertex set $[n]$, denoted as $\Delta_{n-1}^{(k)}$, is shellable.

When working with a shellable complex, one might naturally wonder if it is possible to encounter obstacles while constructing a shelling sequence. A shellable complex $\mC$ is termed \emph{extendably shellable} if every shelling of a subcomplex of $\mC$ can be expanded to form a shelling of $\mC$ itself. In this context, a subcomplex of $\mC$ refers to a simplicial complex where its facets form a subset of the facets of $\mC$.

Shellable complexes are naturally found in various scenarios, but it appears that extendably shellable complexes are less prevalent. Some known examples are: every $2$-dimensional triangulated sphere \cite{dk78}, any $d$-dimensional sphere with $d+3$ vertices \cite{k77}, and rank 3 matroid independence complexes \cite{be94}. However, Ziegler \cite{z98} provided evidence that some simplicial 4-polytopes do not share this attribute. In addition, there are $2$-dimensional complexes on $6$ vertices that are not extendably shellable \cite{mt03}. The main motivation behind a significant part of our research is a question posed by Simon \cite{s94}.

\begin{conjecture}[Simon's Conjecture \cite{s94}]
The complex $\Delta_{n-1}^{(d)}$ is extendably shellable. 
\end{conjecture}

Some small cases have been resolved: the $d=2$ case \cite{be94} and the $d \geq n-3$ case \cite{bpz19,d21,cdgs20}. However, this conjecture in general is still wide open. An equivalent formulation of the above conjecture is the following question.

\begin{question}[Simon's Conjecture]
Given any pure $d$-dimensional shellable simplicial complex $\mC$ on $n$ vertices, can we extend $\mC$ to $\Delta_{n-1}^{(d)}$ adding one facet at a time, while maintaining shellability?
\end{question}



From this perspective, it is of interest to find a large class of shellable complexes that we can extend one facet at a time. There is a notion of $k$-decomposable complexes introduced by \cite{pb80} that interpolates between vertex decomposable complexes and shellable complexes: for $k=0$ it is the class of vertex decomposable complexes and for $k=d$ it equals shellability. In \cite{cdgo22}, the extendability analogous to Simon's conjecture for $0$-decomposable complexes was shown.

\begin{theorem}[\cite{cdgo22}]
Given any pure $d$-dimensional $0$-decomposable simplicial complex $\mC$ on $n$ vertices, we can extend $\mC$ to $\Delta_{n-1}^{(d)}$ adding one facet at a time, while maintaining $0$-decomposability.
\end{theorem}

As $0$-decomposable complexes are shellable, this result answers Simon's conjecture for the specific class of $0$-decomposable complexes. The purpose of this paper is to study the extendability of $1$-decomposable complexes, a class that strictly contains the $0$-decomposable complexes and is contained in the class of shellable complexes.

\newtheorem*{thm:main}{\cref{thm:main}}
\begin{thm:main}
Given any pure $d$-dimensional $1$-decomposable simplicial complex $\mC$ on $n$ vertices, we can extend $\mC$ to $\Delta_{n+d-3}^{(d)}$ one facet at a time, while maintaining $1$-decomposability.
\end{thm:main}

The remainder of the paper is organized as follows. In Section~2 we recall some necessary definitions and discuss some basic results on $k$-decomposable complexes. In Section~3 we come up with some tools and answer the problem in low-dimensional cases. In Section~4 we prove our main result \cref{thm:main}. We end in Section~5 with some discussion and open questions.

\section{Preliminaries}

\subsection{Basic terminology}
A \newword{simplicial complex} $\mC$ on a finite ground set $V$ is a collection of subsets of $V$ that are closed under taking subsets, so that if $F_1 \in \mC$ and $F_2 \subset F_1$ then $F_2 \in \mC$.  The elements of $\mC$ are called \newword{faces}. Note that we do not require $\{v\} \in \mC$ for all $v\in V$. The elements $v \in V$ such that $\{v\} \in \mC$ are called the \newword{vertices} of $\mC$, while the elements $w \in V$ that are not vertices will be called \newword{loops}. Following the convention in \cite{j08}, \newword{void complex} $\emptyset$ is a simplicial complex, distinct from the \newword{empty complex} $\{\emptyset\}$.

 A \newword{facet} of $\mC$ is an element that is maximal under inclusion. The \newword{dimension} of $\mC$ is the largest cardinality minus $1$ of any facet. A simplicial complex $\mC$ is \newword{pure} if all facets have the same cardinality. We use $\langle F_1, \dots, F_k \rangle$ to denote the simplicial complex with facets $F_1,\dots,F_k$. 


\begin{definition}\label{defn:shellable}
 A pure $d$-dimensional simplicial complex $\mC$ is said to be \newword{shellable} if there exists an ordering of its facets $F_1, F_2, \dots, F_s$ such that for all $k = 2,3, \dots, n$, we have that
\[\left ( \bigcup_{i=1}^{k-1} \langle F_i \rangle \right ) \cap \langle F_ k \rangle\]
\noindent
is pure of dimension $d-1$. By convention, the void complex $\emptyset$ and the empty complex $\{\emptyset\}$ are shellable.  \end{definition}

A shellable complex is connected as long as $d \geq 1$. A pure $1$-dimensional simplicial complex can be thought of as a graph (its facets are the $2$-element subsets of $V$, so they can be thought of as the edges of a graph on $V$). We next recall the operations of link, star, and deletion of a face in a simplicial complex.

\begin{definition}
Suppose $\mC$ is a simplicial complex on ground set $V$ and let $F \in \mC$ be a nonempty face.  The  \newword{link},  \newword{star}, and the \newword{deletion} of $F$ are defined as 
\begin{align*}
\lk_F{\mC} &:= \{G \in \mC : G \cap F = \emptyset, G \cup F \in \mC\}, \\
\plk_F{\mC} &:= \{G \in \mC : F \subset G\}, \\
\del_F{\mC} &:= \{G \in \mC: F \not \subset G\}.
\end{align*}
The ground set of $\plk_F{\mC}$ is $V$, whereas the ground set of $\lk_F{\mC}$ is $V \backslash F$.  If $|F| > 1$, then $\del_F{\mC}$ has ground set $V$, and if $F = \{v\}$ then the ground set is $V \backslash \{v\}$.   
\end{definition}

Next, we discuss the definition of $k$-decomposable complexes.

\begin{definition}[\cite{pb80}]
A pure $d$-dimensional simplicial complex $\mC$ is said to be \newword{$k$-decomposable} if $\mC$ is a simplex, or $\mC$ contains a face $F$ such that
\begin{enumerate}
    \item $\dim(F) \leq k$,
    \item both $\del_F{\mC}$ and $\lk_F{\mC}$ are $k$-decomposable, and
    \item $\del_F{\mC}$ is pure (and the dimensions stay the same as that of $\mC$).
\end{enumerate}
A face $F$ that satisfies the third condition is called a \newword{shedding face} of $\mC$.
\end{definition}

The $0$-decomposable complexes are called \newword{vertex decomposable} complexes. Vertex decomposable complexes were introduced in the pure setting by Provan and Billera \cite{pb80} and extended to non-pure complexes by Bj\"orner and Wachs \cite{bw96}.  It is known that any pure vertex decomposable complex is shellable, a fact implied by the following result of Wachs \cite{w99}. 

From the definition it is clear that $k$-decomposable complexes are $(k+1)$-decomposable as well. The notion of $k$-decomposable interpolates between the notion of vertex decomposable and shellable (which can be seen to coincide with $d$-decomposable).

\begin{example}
Let $\mC = \langle 123,124,125,134,136,245,256,346,356,456 \rangle$ (Example V6F10-6 from \cite{mt03}). In \cite{mt03} it is shown that $\mC$ is not vertex decomposable, but one can check that it is $1$-decomposable using $15$ as a shedding face. 
\end{example}
 
\subsection{Combinatorial tools} There is a purely combinatorial way to check if a face is a shedding face, which we will describe below and use throughout the paper. Given two facets $F_1$ and $F_2$ of a pure $d$-dimensional simplicial complex $\mC$, we say that $F_1$ and $F_2$ are \newword{adjacent} if they differ by one vertex, i.e. $|F_1 \cap F_2| = d$.

\begin{lemma}[Gluing criterion]
\label{lem:shedd}
Suppose $\mC$ is a pure simplicial complex, and let $F$ be a face.  Then $F$ is a shedding face if and only if for any $f \in F$ and any $H \in \plk_F{\mC}$, there exists a facet $H'$ in $\mC$ such that $H \cap H' = H \setminus \{f\}$ (so they are adjacent).
\end{lemma}
\begin{proof}
For the forward direction, suppose $F$ is a shedding face of $\mC$. Let $H$ be a facet of $\plk_F{\mC}$ and choose any $f \in F$. Then $H \setminus \{f\}$ is a face of $\del_F{\mC}$. Therefore, there has to be a facet that contains it.

For the other direction, assume for the sake of contradiction that $F$ is not a shedding face. So some facet $J$ of $\del_F{\mC}$ is not a facet of $\mC$. Let $J'$ be the facet of $\mC$ that contains $J$. This $J'$ has to be contained in $\plk_F{\mC}$, so it contains $F$. There is a facet of $\del_F{\mC}$ adjacent to $J'$, which we will denote by $J''$. Then $J''$ strictly contains $J$, which gives us a contradiction.
\end{proof}

The gluing criterion can be used to check if a face is a shedding face. In the remainder of the paper we will always use the gluing criterion instead of checking if $\del_F{\mC}$ is pure.

\begin{definition}
Given a simplicial complex $\mC$ and a set $F$ that has the same dimension as the dimension of $\mC$, we use $\mC + F$ to denote the complex generated by the set of facets of $\mC$ and $F$. Given two simplicial complexes $\mC$ and $\mD$ of equal dimension, we write $\mC \setminus \mD$ as the simplicial complex generated by the facets in $\mC$ but not in $\mD$. When $\mC$ and $\mD$ do not share facets, we write $\mC + \mD$ as the simplicial complex generated by the facets in either of them.
\end{definition}

\begin{remark}
From this point on, we will only deal with pure simplicial complexes, except for $1$-dimensional complexes. Given a simplicial complex of dimension $2$ or higher, unless otherwise stated, it should automatically be assumed that we are dealing with a pure simplicial complex.
\end{remark}

The main reason we consider nonpure $1$-dimensional complexes is because we eventually deal with graphs that can have isolated vertices.

We end the section with the following tools.

\begin{lemma}
\label{lem:1decordering}
Let $\mC$ be a pure $k$-decomposable complex. Then we can order the facets $F_1,\dots,F_t$ such that $\langle F_1,\dots,F_i \rangle$ for $1 \leq i \leq t$ is $k$-decomposable.
\end{lemma}
\begin{proof}
We use induction on the number of facets of $\mC$. The base case is where $\mC$ has exactly one facet, in which case the argument is trivial. Now assume for the sake of induction that the claim holds for complexes that have a smaller number of facets than $\mC$.

Let $F$ be the shedding face of $\mC$. From the induction hypothesis, we can order the facets in $\del_F{\mC}$. We can also order the facets in $\plk_F{\mC}$. Then we combine the orders by going over the facets of $\del_F{\mC}$ first.
\end{proof}

\begin{example}
    Consider the pure simplicial complex \[\langle 123, 234, 134, 135, 145, 245\rangle.\]
    We have $45$ as a shedding face, so we put the facets $\{123, 234, 134,135\}$ first, then put the facets $\{145,245\}$. One can check that $\langle 123 \rangle, \dots, \langle 123,234, 134,135,145,245 \rangle$ are all $1$-decomposable.  
\end{example}

As can be observed from the gluing criterion \cref{lem:shedd}, when $F$ is a shedding face of $\mC$, after adding more facets that do not contain $F$, the face $F$ is still a shedding face. 

\begin{lemma}
\label{lem:cheatcheckshedding}
Let $\mC$ be a pure $k$-decomposable complex such that $F$ is either a shedding face or not a face of $\mC$ at all. Let $F_1,\dots,F_t$ be facets not in $\mC$ that do not contain $F$ (we do allow $t=0$, which corresponds to such set being empty) and let $F_{t+1},\dots,F_q$ be facets not in $\mC$ that contain $F$. If $\mC + \langle F_1,\dots,F_q \rangle$ has $F$ as a shedding face, then $F$ is a shedding face in $\mC + \langle F_1,\dots,F_i \rangle$ for any $1 \leq i \leq q$.
\end{lemma}
\begin{proof}
As mentioned in the paragraph just before the statement of this lemma, adding facets that do not contain $F$ to $\mC$ does not affect $F$ being a shedding face. If $\mC + \langle F_1,\dots,F_q\rangle$ has $F$ as a shedding face, that means all facets $F_{t+1},\dots,F_q$ satisfy the gluing criterion \cref{lem:shedd}, so $F$ is a shedding face in $\mC + \langle F_1,\dots,F_i \rangle$ for any $i$.
\end{proof}

The above lemma tells us the following. When we extend a complex incrementally one facet at a time, as long as we put the facets that avoid $F$ first, we can check whether $F$ was a shedding face throughout the entire process just by checking it on the final product.

\section{Low-dimensional cases}

In this section, we develop some tools that we need to prove the main result. Most of them will be on how $1$-decomposable complexes in dimension $1$ or $2$ behave. 

We start with a lemma that allows us to relate the properties of $\plk_F{\mC}$ and $\lk_F{\mC}$.

\begin{lemma}
\label{lem:pad}
Let $\mC = \langle F_1, \dots, F_t \rangle$ be a pure simplicial complex and let $H$ be a set disjoint from the vertex set of $\mC$. Set $\mD = \langle F_1 \cup H, \dots, F_t \cup H \rangle$. Then $\mC$ is $k$-decomposable if and only if $\mD$ is, and $F$ is a shedding face of $\mC$ if and only if $F$ is a shedding face of $\mD$.
\end{lemma}
\begin{proof}
We use induction on the number of facets of $\mC$. In the base case where there is only one facet, the argument is trivial. Now assume for the sake of induction that the claim holds for complexes with fewer facets than $\mC$.

Using the gluing criterion \cref{lem:shedd}, we can see that $F$ is a shedding face of $\mC$ if and only if $F$ is a shedding face of $\mD$. Moreover, the induction hypothesis shows that $\del_F{\mC}$ is $k$-decomposable if and only if $\del_F{\mD}$ is, and a similar argument holds for the link as well. This finishes the proof.
\end{proof}

Hence, whenever we are in a scenario where we have to show that $\plk_F{\mC}$ is $1$-decomposable, we can use $\lk_F{\mC}$ instead and vice versa. 

\begin{example}
Consider the pure simplicial complex
$$\langle 1234,2345,1237,1238,2348,2358,2378,1235,2347,2357 \rangle.$$ To check whether this is $1$-decomposable, we can instead check whether
$$\langle 14,45,17,18,48,58,78,15,47,57 \rangle$$ is $1$-decomposable or not. Since this is a connected graph, \cref{lem:1dconn} will show us that it is indeed $1$-decomposable.
\end{example}

Given a pure simplicial complex $\mC$, we will use $\skel(\mC)$ to denote its \newword{$1$-dimensional skeleton}: the simplicial complex generated by all $1$-dimensional faces of $\mC$. Since this is a $1$-dimensional complex, we can treat it as a graph. 

\begin{lemma}[Section 7.2, \cite{bj92}]
\label{lem:skelconn}
Let $\mC$ be a pure simplicial complex that is shellable. Then $\skel(\mC)$ is a connected graph. 
\end{lemma}

We start with a simple observation on how the skeleton changes as we take the deletion or the link of a shedding face. 

\begin{lemma}
\label{lem:skeldecompanal}
Let $\mC$ be a pure simplicial complex with $F$ as a shedding face. 

When $F = w$, we have:
\begin{itemize}
    \item $\skel(\del_F{\mC})$ is the graph obtained from $\skel(\mC)$ by removing the vertex $w$ and all edges incident to it.
    \item If the dimension of $\lk_F{\mC}$ is nonzero, then $\skel(\lk_F{\mC})$ is a subgraph (with the same vertex set) of the induced subgraph of $\skel(\mC)$ on the vertices adjacent to $w$.
\end{itemize}

When $F = ab$, we have:
\begin{itemize}
    \item $\skel(\del_F{\mC})$ is the graph obtained from $\skel(\mC)$ by removing only the edge $ab$.
    \item If the dimension of $\lk_F{\mC}$ is nonzero, then $\skel(\lk_F{\mC})$ is a subgraph (whose vertex set may be strictly smaller) of the induced subgraph of $\skel(\mC)$ on the vertices adjacent to both $a$ and $b$.
\end{itemize}
\end{lemma}
\begin{proof}
Take any $J \not \subseteq F$ in $\skel(\mC)$. There is a facet $H$ that contains it in $\mC$. If $H \in \del_F{\mC}$, then $J$ is contained in $\skel(\del_F{\mC})$ from the definition of a $1$-skeleton. Otherwise, we have $H \in \plk_F{\mC}$, and from the gluing criterion \cref{lem:shedd} we have some adjacent facet in $\del_F{\mC}$ that contains $J$, so again we have $J$ contained in $\skel(\del_F{\mC})$.

When $F=w$, then all edges with $w$ as an endpoint come only from the facets of $\plk_F{\mC}$. Hence $\skel(\lk_F{\mC})$ has the same vertex set as the induced subgraph of $\skel(\mC)$ on vertices adjacent to $w$.

When $F=ab$, some edges with $a$ or $b$ as an endpoint may appear in $\skel(\mC)$ but not in $\skel(\lk_F{\mC})$: this happens when there is a facet containing $ax$ and a facet containing $bx$ but no facet containing $abx$ in $\mC$.
\end{proof}

\begin{example}
    Consider a $1$-decomposable simplicial complex and its skeleton 
    \begin{align*}
    \mC &= \langle 123, 134, 145, 345\rangle, \\ 
    \skel(\mC) &= \{12, 13, 14, 15, 23, 34, 35, 45\}.
    \end{align*}
    We have a shedding face $F = 5$. Then we get 
    \begin{align*}
    \del_F{\mC} &= \langle123, 134\rangle, \\ 
    \skel(\del_F{\mC}) &= \{12, 13, 14, 23, 34\}, \\
    \lk_F{\mC} &= \langle14, 34\rangle, \\ 
    \skel(\lk_F{\mC}) &= \{14, 34\}.
    \end{align*}
\end{example}

Next, we show a criterion that simplifies the process of checking whether a given $1$-dimensional complex is $1$-decomposable or not.

\begin{lemma}
\label{lem:1dconn}
Let $\mC$ be a pure $1$-dimensional simplicial complex. Then $\mC$ is $1$-decomposable if and only if $\skel(\mC)$ (which is equal to $\mC$) is a connected graph.
\end{lemma}
\begin{proof}
Recall that $d$-decomposability for $d$-dimensional complexes is equivalent to shellability. A graph is shellable if and only if it is connected [Problem 34, \cite{kp03}], which concludes the proof.
\end{proof}

From the lemma, given a pure $1$-dimensional $1$-decomposable complex, it is simple to find a facet to add and keep it $1$-decomposable: find any edge that preserves the connectivity of the graph after adding it. For $2$-dimensional cases, we have the following.

\begin{lemma}
\label{lem:2dk3ext}
Consider a pure $2$-dimensional $1$-decomposable complex $\mC$. If there is an edge $e$ missing from the graph $\skel(\mC)$, but it creates a $3$-clique if we add it to the graph, then we can find some facet $H$ such that $\mC + H$ is $1$-decomposable and $\skel(\mC+H) = \skel(\mC) + e$.
\end{lemma}
\begin{proof}
Let $e=ab$ be the edge we add that forms a newly created $3$-clique $abc$ in $\skel(\mC) + e$. We are going to use $H=abc$ to add to $\mC$, using $e$ as the shedding face.

As $ac,bc$ are the edges of $\skel(\mC)$, there is a facet of $\mC$ that contains $ac$ and a facet that contains $bc$. Since these facets are adjacent to $abc$, we conclude that $e$ is a shedding face from \cref{lem:shedd}. 

We have $\del_e(\mC+H) = \mC$ is $1$-decomposable. We see that $\lk_e(\mC+H)$ is $1$-decomposable since it is $0$-dimensional. This concludes the proof that $\mC + H$ is $1$-decomposable.
\end{proof}

The above lemma gives us a way to expand the skeleton of a $2$-dimensional $1$-decomposable complex while maintaining $1$-decomposability. 

\begin{example}
Consider $\mC = \langle 123,134,145 \rangle$, a $1$-decomposable complex that has skeleton $\{12,13,14,15,23,34,45\}$. Then $25$ is an edge that is missing and forms a $3$-clique $125$ if added in. Thus we can create $\mC + H = \langle 123,134,145,125 \rangle$ that is still $1$-decomposable and has skeleton $\{12,13,14,15,23,25,34,45\}$.
\end{example}

Next, we define an operator we will be using. 
\begin{definition}
\label{def:full}
Given a graph $G$, $\full^{d}(G)$ is a pure $d$-dimensional simplicial complex generated by all $(d+1)$-cliques of $G$. Given a pure simplicial complex $\mC$, we let $\full^{d}(\mC)$ stand for $\full^{d}(\skel(\mC))$.
\end{definition}

When we omit $d$ and just write $\full(\skel(\mC))$, it means that we are taking $\full^d(\skel(\mC))$ when $\mC$ is a $d$-dimensional complex.

\begin{example}
Take $\mC = \langle 123,134,124 \rangle$. Then we have $\skel(\mC) = \{12,13,14,23,24,34\}$. Hence, $\full(\mC) = \full^2(\mC) = \langle 123,134,124,234 \rangle$.
\end{example}

Note that given a pure $1$-dimensional complex, it is already equal to its full, as the complex itself is equal to its skeleton as a graph. For a pure $2$-dimensional complex, we show in the following that one can always add a facet one by one while maintaining $1$-decomposability and its skeleton (unless the original complex is already equal to its full).

\begin{lemma}
\label{lem:2dfull}
Consider $\mC$ a pure $2$-dimensional $1$-decomposable complex. It is possible to extend $\mC$ to $\full(\mC)$ incrementally by adding one facet at a time, while ensuring that it remains $1$-decomposable. 
\end{lemma}
\begin{proof}
 We are going to prove the claim by induction on $|\skel(\mC)|$. The base case of $|\skel(\mC)|=3$ is when we have exactly one facet in $\mC$, where the argument is trivial. For the sake of induction, assume that we have the claim for complexes with a skeleton smaller than $\mC$.

 Let $F$ denote the shedding face of $\mC$. The induction hypothesis tells us that we can extend $\del_F{\mC}$ to $\full(\del_F{\mC})$ one facet at a time while maintaining $1$-decomposability. We start by showing that $F$ is also a shedding face for $\full(\mC)$. We perform a case-by-case analysis based on the dimension of $F$.

 When $F = w$, take any facet $wab$ in $\plk_F{\full(\mC)}$. Since $ab$ is an edge of $\skel(\mC)$, it is also contained in $\skel(\del_F{\mC})$ due to \cref{lem:skeldecompanal}. This implies that there is some facet in $\del_F{\mC}$ that contains $ab$, and that facet is adjacent to $wab$.
 
 When $F=ab$, take any facet $abx$ in $\plk_F{\full(\mC)}$. Since $ax,bx$ are the edges in $\skel(\mC)$, they are also in $\skel(\del_F{\mC})$ due to \cref{lem:skeldecompanal}. This implies that there is some facet in $\del_F(\mC)$ that contains $ax$ and some facet that contains $bx$, and they are adjacent to $abx$. 

 As the facets of $\del_F{\mC}$ are also facets in $\full(\mC)$, we have shown that $\plk_F{\full(\mC)}$ glues nicely in $\full(\mC)$ along $F$. Then \cref{lem:shedd} tells us that $F$ is a shedding face of $\full(\mC)$. Moreover, from \cref{lem:cheatcheckshedding}, as we extend from $\mC$ to $\full(\mC)$ one facet at a time, $F$ stays as a shedding face throughout this process as long as we add facets not containing $F$ first.

As $F$ is a shedding face of $\full(\mC)$, we get $\full(\del_F{\mC}) = \del_F{\full(\mC)}$ due to \cref{lem:skeldecompanal}. Therefore, after extending $\del_F{\mC}$ to $\full(\del_F{\mC})$, our goal now is to add facets to $\plk_F{\mC}$ to extend it to $\plk_F{\full(\mC)}$. From \cref{lem:pad}, we can instead focus on extending $\lk_F{\mC}$ to $\lk_F{\full(\mC)}$. We again do a case-by-case analysis on the dimension of $F$.

When $F = w$, due to \cref{lem:skelconn} and \cref{lem:skeldecompanal}, we have that $\lk_F{\mC}$ is a connected subgraph of a connected graph $\lk_F{\full(\mC)}$. We can add edges one by one to extend $\lk_F{\mC}$ to $\lk_F{\full(\mC)}$, while maintaining the connectivity of the graph (and $1$-decomposability as well). 

When $F=ab$, we have that $\lk_F{\mC}$ is $0$-dimensional. We can extend $\lk_F{\mC}$ to $\lk_F(\full(\mC))$ adding vertices one at a time.

We have shown that we can extend $\del_F{\mC}$ to $\full(\del_F{\mC}) = \del_F{\full(\mC)}$ incrementally adding one facet at a time, while keeping it $1$-decomposable. Then we can extend $\plk_F{\mC}$ to $\plk_F{\full(\mC)}$ one facet at a time, while keeping it $1$-decomposable. We have also shown that $F$ is a shedding face throughout this process. This finishes the proof.
\end{proof}

In the following, we provide an example that illustrates the process used in the proof of \cref{lem:2dfull}.

\begin{example}
    Consider the 1-decomposable simplicial complex \[\mC = \langle 124, 134, 235, 245, 345\rangle.\] We have a shedding face $F = 1$.  We have \[\full(\mC) = \langle 124,134, 235, 245, 345, 234, 123\rangle.\] 
    First extend $\del_F{\mC} = \langle 235,245,345 \rangle$ to $\full(\del_F{\mC})$ by adding the facet $234$.
    
    Then extend $\lk_F{\mC} = \langle 24,34 \rangle$ to $\lk_F{\full(\mC)}$ by adding the facet (edge) $23$.

    We end up adding the facets in order $234,123$.
\end{example}

The approach we used in the proof of \cref{lem:2dfull} is essentially the following: we extend the deletion first until we cannot anymore (while maintaining the skeleton), then extend the remaining parts including the star again until we cannot anymore. This induction-based technique will be used in the vast majority of our proofs throughout the paper.


\section{Main result}

In this section, we will state and prove our main result. The results in the previous section on lower-dimensional complexes will be used as a base case in a lot of the induction-based proofs that we carry out in this section.

We start with the following definition.
\begin{definition}
\label{def:coned}
For $d \geq 2$, we say that a pure $d$-dimensional complex $\mC$ is \newword{fully coned with respect to $H$}, where $H$ is a subset of the vertex set, if it satisfies the following properties:
\begin{itemize}
    \item for each $h \in H$, $h$ is adjacent to every other vertex in $\skel(\mC)$,
    \item any $(d+1)$-clique in $\skel(\mC)$ is a facet of $\mC$.
\end{itemize}
\end{definition}

When $H = \emptyset$, saying that $\mC$ is fully coned with respect to $\emptyset$ is the same as saying that $\mC = \full(\mC)$. When $H$ is large enough, we get the following property. 

\begin{lemma}
\label{lem:lowcliquecontain}
For $d \geq 2$, let $\mC$ be a pure $d$-dimensional complex that is fully coned with respect to $H$ where $|H| \geq d-2$. Then, for any $d$-clique in $\skel(\mC)$ there is some facet of $\mC$ that contains it.
\end{lemma}
\begin{proof}
If our $d$-clique does not contain all the vertices of $H$, we can simply attach a missing vertex from $H$ to get a $(d+1)$-clique, which is a facet in $\mC$ due to the definition of a fully coned complex.

In the case where our $d$-clique contains all vertices of $H$, it is of the form $H \cup F$ where $H \cap F = \emptyset$ and $|F| \leq 2$. As $F$ is a face of $\skel(\mC)$, there is some facet in $\mC$, and a $(d+1)$-clique corresponding to it in $\skel(\mC)$, which contains $F$. We can find some $w$ in this clique not in $H \cup F$, and $H \cup F \cup \{w\}$ is a $(d+1)$-clique in $\skel(\mC)$, which is a facet in $\mC$ due to the definition of a fully coned complex.
\end{proof}

On a related note, we show an important property that fully coned graphs, defined below, satisfy. The reason we define $G^H$ separately below instead of using the above framework for $1$-dimensional cases is because we want to allow isolated vertices in $G$, which makes $G$ not a pure simplicial complex.

\begin{lemma}
\label{lem:arbgraphcone}
Let $G$ be an arbitrary graph on vertex set $V$, potentially with isolated vertices. When $H$ is a set of cardinality at least $d$ and $H \cap V = \emptyset$, we let $G^H$ denote the graph we get by connecting each vertex of $H$ to every other vertex of $H \cup V$. Then $\full^{d}(G^H)$ is $1$-decomposable.
\end{lemma}
\begin{proof}
We will first use induction on the dimension. The base case of $d=1$ is trivial, since $G^H$ is a connected graph and we can use \cref{lem:1dconn}.  Hence, we may assume for the sake of induction that our claim holds for smaller $d$'s.

Next, we use induction on $|V|$. The base case of $|V|=1$ is trivial: $\full^{d}(G^H)$ is a $d$-skeleton of a simplex on $|H|+1$ vertices. Now assume for the sake of induction that our claim holds for smaller $V$'s.

Take any vertex $v \in V$. We first show that $v$ is a shedding face in $\full^{d}(G^H)$. Take any facet $v c_1 \dots c_{d}$ of $\plk_v{\full^{d}(G^H)}$. If this facet contains $H$, then we can simply replace $v$ with any other vertex in $V$ to obtain an adjacent facet in $\full^{d}(G^H)$. If not, then we can replace $v$ with any unused vertex of $H$, to obtain an adjacent facet in $\full^{d}(G^H)$. From \cref{lem:shedd}, we can see that $v$ is a shedding face.

Use $G_1$ to denote the graph obtained from $G$ by deleting the vertex $v$. Since $v$ is a shedding face, from \cref{lem:skeldecompanal}, we get $\del_v{\full^{d}(G^H)} =  \full^{d}(G_1^H)$. It follows from the induction hypothesis that $\full^{d}(G_1^H)$ is $1$-decomposable, since it has a smaller vertex set.

 Looking at $\lk_v{\full^{d}(G^H)}$, since all vertices of $H$ are adjacent to $v$, we have $$\lk_v{\full^{d}(G^H)} = \full^{d-1}(G_2^H),$$ where $G_2$ is obtained from $G$ by taking the induced subgraph on vertices that are adjacent to $v$ (i.e. we are taking the link of $v$ in $G$). The complex $\full^{d-1}(G_2^H)$ is $1$-decomposable due to the induction hypothesis since it has a lower dimension.

We have shown that $v$ is a shedding face and that both $\del_v{\full^{d}(G^H)}$ and $\lk_v{\full^{d}(G^H)}$ are $1$-decomposable. This concludes the proof of the claim.
\end{proof}

We look at an example below.

\begin{example}
Take $G$ on vertex set $[5] = \{1,\dots,5\}$ with edges $12$ and $34$ (so $5$ is an isolated vertex). Let $H = \{a,b\}$. Then $$\full^{2}(G^H) = \langle ab1,ab2,ab3,ab4,ab5,a12,b12,a34,b34 \rangle$$ is $1$-decomposable, and we can use any vertex of $[5]$ as a shedding face.    
\end{example}

We go over the high-dimensional version of \cref{lem:2dk3ext}. 

\begin{lemma}
\label{lem:cliquemissedgeextend}
For $d \geq 2$, let $\mC$ be a pure $d$-dimensional $1$-decomposable complex that is fully coned with respect to $H$ such that $|H| \geq d-2$. Let $e$ be an edge missing from $\skel(\mC)$ such that, once added to $\skel(\mC)$, it creates a new $(d+1)$-clique. Then we can extend $\mC$ one facet at a time to a fully coned complex with respect to $H$ with skeleton $\skel(\mC) + e$ while maintaining the property of being $1$-decomposable.
\end{lemma}
\begin{proof}
We induct on $d$. The base case where $d=2$ follows from \cref{lem:2dk3ext} and \cref{lem:2dfull}. Now assume for the sake of induction that the claim holds for smaller $d$ than what we have now.

Let $\mD$ denote the fully coned complex with respect to $H$, that has skeleton equal to $\skel(\mC) + e$. The facets that we are trying to add to $\mC$ to reach $\mD$ are the $(d+1)$-cliques in $\skel(\mC) + e$ that use $e=ab$. There are two things that we need to show: first, we can use $e$ as the shedding face throughout this process, and second, we can order the facets in $\mD \setminus \mC$ as $F_1,\dots,F_t$ so that $\langle F_1,\dots, F_i \rangle$ is $1$-decomposable for each $1 \leq i \leq t$.

We first show that $e$ is a shedding face of $\mD$. Let $c_1\dots c_{d-1} ab$ denote the $(d+1)$-clique in $\skel(\mC) + e$. Then $c_1\dots c_{d-1} a$ and $c_1 \dots c_{d-1} b$ are $d$-cliques in $\skel(\mC)$, so for each of them we have a facet in $\mC$ that contains it due to \cref{lem:lowcliquecontain}, and those facets are adjacent to $c_1\dots c_{d-1} ab$. From \cref{lem:shedd} and \cref{lem:cheatcheckshedding}, we can see that $e$ is a shedding face throughout the process of extending $\mC$ to $\mD$.

Next, we show that $\langle F_1,\dots,F_t \rangle$ forms a $1$-decomposable complex. We are looking at an induced subgraph of $\skel(\mC)$ on the set of vertices adjacent to both $a$ and $b$, and taking all the $(d-1)$-cliques and attaching $ab$ to them to get all the facets in $\mD \setminus \mC$. This is a $1$-decomposable complex due to \cref{lem:arbgraphcone} and \cref{lem:pad}. From this we have an ordering of the facets $F_1,\dots,F_t$ that we desire due to \cref{lem:1decordering}.
\end{proof}

We now go over the high-dimensional version of \cref{lem:2dfull}. The proof will mirror that of \cref{lem:2dfull}, with $\cocl{\mC}$ (defined below) instead playing the role of $\full(\mC)$.

\begin{proposition}
\label{prop:1dectococl}
For $d \geq 2$, consider $\mC$ a pure $d$-dimensional $1$-decomposable complex on vertex set $V$ and $H$ a set of vertices of size $\geq d-2$ where $\mC$ is disjoint from $H$ ($V \cap H = \emptyset$). It is possible to extend $\mC$ to $\cocl{\mC} := \full^{d}((\skel(\mC))^H)$ incrementally by adding one facet at a time, while ensuring that it remains $1$-decomposable.
\end{proposition}
\begin{proof}
We are going to prove the claim by induction. We will first induct on the dimension. The base case of $d=2$ comes from \cref{lem:2dfull}. For sake of induction assume we have the claim for complexes that have dimension less than $d$.

Then we use induction on the size of $\skel(\mC)$. The base case is when $|\skel(\mC)| = {d+1 \choose 2}$: our complex consists only of one facet. Using Theorem 2.9 of \cite{cdgo22}, we can extend this to the $d$-dimensional skeleton of a simplex on $d+1 + |H|$ vertices adding one facet at a time, while maintaining $0$-decomposability. Now assume for the sake of induction that we have the claim for complexes with a skeleton 
smaller than $\mC$.

Let $F$ denote the shedding face of $\mC$. Then we can extend $\del_F{\mC}$ to $\cocl{(\del_F{\mC})}$ one facet at a time while maintaining $1$-decomposability due to the induction hypothesis since $\del_F{\mC}$ has a smaller skeleton. All facets of $\cocl{(\del_F{\mC})}$ are facets in $\cocl{\mC}$ as well. We start by showing that $F$ is a shedding face for $\cocl{\mC}$. We do a case-by-case analysis on dimension of $F$.


When $F = w$, take any facet $w c_1 \dots c_{d}$ in $\plk_F{\cocl{\mC}}$. Then $c_1 \dots c_d$ is a $d$-clique in $\skel(\cocl{(\del_F{\mC})})$, so there is a facet containing it in $\cocl{(\del_F{\mC})}$ from \cref{lem:lowcliquecontain}, and hence in $\cocl{\mC}$ as well. 


When \( F = ab \), take any facet \( abc_1 \dots c_{d-1} \) in \( \plk_F{\cocl{\mC}} \). Then \( ac_1 \dots c_{d-1} \) is a \( d \)-clique in \( \skel(\cocl{(\del_F{\mC})}) \), so there is a facet containing it in \( \cocl{(\del_F{\mC})} \) from \cref{lem:lowcliquecontain}, and hence in \( \cocl{\mC} \). The same reasoning applies to \( bc_1 \dots c_{d-1} \).


 From \cref{lem:shedd}, we obtain that $F$ is a shedding face of $\cocl{\mC}$. Moreover, from \cref{lem:cheatcheckshedding}, as we extend from $\mC$ to $\cocl{\mC}$ one facet at a time, $F$ stays as a shedding face throughout this process as long as we add facets not containing $F$ first.

As $F$ is a shedding face of $\cocl{\mC}$, we get $\cocl{(\del_F{\mC})} = \del_F{\cocl{\mC}}$ due to \cref{lem:skeldecompanal}. Therefore, after extending $\del_F{\mC}$ to $\cocl{(\del_F{\mC})}$, our goal now is to add facets to $\plk_F{\mC}$ to extend it to $\plk_F{\cocl{\mC}}$. From \cref{lem:pad}, we can instead focus on extending $\lk_F{\mC}$ to $\lk_F{\cocl{\mC}}$. We use the induction hypothesis to extend $\lk_F{\mC}$ to $\cocl{(\lk_F{\mC})}$, since it has a lower dimension. For extending $\cocl{(\lk_F{\mC})}$ to $\lk_F{\cocl{\mC}}$, we again do a case-by-case analysis on the dimension of $F$.

When $F = w$, from \cref{lem:skelconn} and \cref{lem:skeldecompanal}, the graph $\skel(\cocl{(\lk_F{\mC})})$ is a connected subgraph of connected graph $\skel(\lk_F{\cocl{\mC}})$ with the same vertex set. Any missing edge, if put in, would form a $d$-clique with vertices of $H$ since each vertex of $H$ is adjacent to all other vertices. We use \cref{lem:cliquemissedgeextend} to extend, adding all such edges to the skeleton, which will allow us to reach $\lk_F{\cocl{\mC}}$ while maintaining $1$-decomposability.

When $F=w_1w_2$, we have that $\skel(\cocl{(\lk_F{\mC})})$ is a connected subgraph inside a connected graph $\skel(\lk_F{\cocl{\mC}})$ from \cref{lem:skelconn} and \cref{lem:skeldecompanal} (the vertex set might be strictly smaller). For each missing vertex $v$, we add the facets $\{v\} \cup H'$, where $H'$ ranges over all subsets of $H$ of size $d-2$, using $v$ as the shedding vertex, adding $v$ to our skeleton. After that, we again use \cref{lem:cliquemissedgeextend} repeatedly to fill in all the missing edges in the skeleton (the edge, together with vertices of $H$, would form a $(d-1)$-clique) to reach $\lk_F{\cocl{\mC}}$ while maintaining $1$-decomposability.

We have shown that we can first extend $\del_F{\mC}$ to $\cocl{(\del_F{\mC})}$ incrementally, adding one facet at a time while keeping it $1$-decomposable. Then we extend $\plk_F{\mC}$ to $\plk_F{\cocl{\mC}}$ one facet at a time while maintaining $1$-decomposability. We have also shown that $F$ is a shedding face throughout this process. This finishes the proof of the claim.
\end{proof}

The following example illustrates the procedure used in the proof of \cref{prop:1dectococl}.

\begin{example}
\label{ex:fullyconeex1}
Consider the complex 
\[
\mC = \langle 1234, 1245, 1345, 2345, 1246, 1256, 1237, 1247, 1347, 3457 \rangle,
\]
with shedding face \(23\). We will use $H = \{8\}$. We have 
\[
\del_{23} \mC = \langle 1245, 1345, 1246, 1256, 1247, 1347, 3457 \rangle.
\]
Then 
\[
\cocl{\del_{23} \mC} =
\begin{aligned}[t]
\langle &1245, 1246, 1247, 1248, 1256, 1257, 1258, 1268, 1278, 1345, 1347, 1348,\\
  &1357, 1358, 1378, 1456, 1457, 1458, 1468, 1478, 1568, 1578,\\
  &2456, 2457, 2458, 2468, 2478, 2568, 2578, 3457, 3458, 3478,\\
  &3578, 4568, 4578 \rangle.
\end{aligned}
\]
We have
\[
\lk_{23} \mC = \langle 14,45,17 \rangle.
\]
Then
\[
\cocl{(\lk_{23} \mC)} = \langle 14,45,17,18,48,58,78 \rangle.
\]
After that we extend it to (take edges where both endpoints are adjacent to $2$ and $3$ in $\skel(\mC)$)
\[
\langle 14,45,17,18,48,58,78,15,47,57 \rangle.
\]
At the end, we get
\[
\cocl{\mC} = \cocl{\del_{23} \mC} + \langle 1234,2345,1237,1238,2348,2358,2378,1235,2347,2357 \rangle.
\]

\end{example}

We next show how to extend the fully coned complexes obtained from \cref{prop:1dectococl}.


\begin{proposition}
\label{prop:fulltosimp}
Let $\mC$ be a pure $d$-dimensional $1$-decomposable complex that is fully coned with respect to $H$, on vertex set $[n] \cup H$, where $[n] \cap H = \emptyset$ and $|H| = d-2$. Moreover, we require that in the graph $\skel(\mC)$, its induced subgraph on $[n]$ be connected. Then we can extend $\cocl{\mC}$ to $\Delta_{n+d-3}^{(d)}$ one facet at a time while maintaining $1$-decomposability.
\end{proposition}
\begin{proof}
When we add an edge $e$ missing from $\skel(\mC)$ that creates a new $(d+1)$-clique to $\skel(\mC)$, it maintains the property that the induced subgraph on $[n]$ is connected: this is because since $|H| = d-2$, the $(d+1)$-clique involves a $3$-clique on $[n]$. We are adding an edge that creates a new $3$-clique on $[n]$, so the induced subgraph on $[n]$ will still be connected after the addition of $e$.

From \cref{lem:cliquemissedgeextend}, it is enough to show the following claim: Given $\skel(\cocl{\mC})$, we can continue to add edges that are the unique missing edge of some $(d+1)$-clique and reach the complete graph on $[n] \cup H$.

Since $\mC$ is fully coned with respect to $H$, in $\skel(\mC)$, the vertices of $H$ are adjacent to every other vertex. Then it is enough to show the following claim on the induced subgraph on $[n]$: Given a connected graph on $n$ vertices, we can keep adding edges that are the unique missing edge of some $3$-clique, and we can reach the complete graph on $n$ vertices.

Given a connected graph, we can check the distance between any pair of vertices. If not all the distances are $1$, then there must exist a pair where the distance is $2$. Then we have a $3$-clique with a missing edge and we can fill that in. Repeating this process, we reach a graph where the distance between any pair of vertices is $1$, which is a complete graph.
\end{proof}

The example below illustrates the procedure used in the proof of \cref{prop:fulltosimp}.

\begin{example}
\label{ex:missingedgeextend}    
We consider the fully coned complex obtained in \cref{ex:fullyconeex1}:
\[
\mC = 
\begin{aligned}[t]
\langle &1234, 1235, 1237, 1238, 1245, 1246, 1247, 1248, 1256, 1257, 1258, 1268, 1278, \\
  &1345, 1347, 1348, 1357, 1358, 1378, 1456, 1457, 1458, 1468, 1478, 1568, 1578, \\
  &2345, 2347, 2348, 2357, 2358, 2378, 2456, 2457, 2458, 2468, 2478, 2568, 2578, \\
  &3457, 3458, 3478, 3578, 4568, 4578 \rangle.
\end{aligned}
\]

Its \(1\)-skeleton is the graph
\[
\begin{aligned}[t]
\langle &12, 13, 14, 15, 16, 17, 18, 23, 24, 25, 26, 27, 28, \\
  &34, 35, 37, 38, 45, 46, 47, 48, 56, 57, 58, 68, 78 \rangle.
\end{aligned}
\]
Notice that $67$ is missing and forms a $3$-clique together with $68$ and $78$. The newly added facets using $67$ from \cref{lem:cliquemissedgeextend} are
\[
\begin{aligned}
\langle 1267, 1467, 1567, 1678, 2467, 2567, 2678, 4567, 4678, 5678 \rangle.
\end{aligned}
\]
This gives us a fully coned complex on the new skeleton with the edge $67$ added. We can repeat this procedure, adding all the missing edges from $K_8$ in the skeleton, to reach $\Delta_7^{(3)}$.
\end{example}

Now we present and prove the main result of the paper.

\begin{theorem}
\label{thm:main}
Let $\mC$ be a pure $d$-dimensional $1$-decomposable complex with vertex set $[n]$. Then we can order the facets of $\Delta_{n+d-3}^{(d)} \setminus \mC = \{F_1,\dots,F_t\}$ so that $\mC + \langle F_1,\dots,F_i \rangle$ is $1$-decomposable for all $1 \leq i \leq t$.
\end{theorem}
\begin{proof}
We extend $\mC$ to $\cocl{\mC}$ using \cref{prop:1dectococl}. Then we extend it to $\Delta_{n+d-3}^{(d)}$ using \cref{prop:fulltosimp}.
\end{proof}

\section{Further questions}

In this section, we propose some questions on the topic.

Our main result \cref{thm:main} required the introduction of $d-2$ new vertices in order to be able to extend an arbitrary pure $d$-dimensional $1$-decomposable complex to the skeleton of a simplex. It is natural to ask if it is possible to do this without allowing the introduction of new vertices (or fewer vertices).

\begin{question}
Can we replace $\Delta_{n+d-3}^{(d)}$ with $\Delta_{n-1}^{(d)}$ in \cref{thm:main}?
\end{question}

The main obstacle when one tries to use a similar approach to what we have done in the paper is that the skeleton of link can be very small. As mentioned in \cref{lem:skeldecompanal}, the skeleton of the deletion is very close to the skeleton of the original complex, whereas for the skeleton of the link, we do not have much control on how big or small it can be compared to the skeleton of the original. 

This means that when we try to extend $\lk_F{\mC}$ to $\lk_F(\full(\mC))$, it is quite difficult. In this paper, we bypassed this issue by introducing several vertices to cone and use \cref{lem:cliquemissedgeextend}, which is essentially using \cref{lem:2dk3ext} due to the existence of $H$. However, we cannot apply such a method if we are not allowed to introduce new vertices.

In \cite{dnsvw24}, it was shown that a complex with a number of facets above a certain threshold is guaranteed to be $0$-decomposable. We could ask if one can come up with a similar threshold for $1$-decomposable complexes:

\begin{question}
Can we come up with a threshold based on the number of vertices and dimension of a $1$-decomposable complex, such that if the number of facets is above that threshold (lower than the threshold for $0$-decomposable), the complex is guaranteed to be $1$-decomposable?
\end{question}

Another natural question to ask is if we can use a similar method to tackle $k$-decomposable complexes.

\begin{question}
Let $\mC$ be a $d$-dimensional $k$-decomposable complex with vertex set $[n]$. Can we find some $\psi(d,k)$ such that we can order the facets of $\Delta_{n-1 + \psi(d,k)}^{(d)} \setminus \mC = \{F_1,\dots,F_t\}$ so that $\mC + \langle F_1,\dots,F_i \rangle$ is $k$-decomposable for all $1 \leq i \leq t$?
\end{question}

As mentioned in the Introduction, since $d$-decomposable is equivalent to shellable, this would be another direction towards resolving Simon's conjecture.


\subsection*{Acknowledgments}
This research was carried out primarily during the 2025 Honors Summer Math Camp at Texas State University. The authors appreciate the support from the camp and also thank Texas State University for providing support and a great working environment. The authors also thank the participants of the 2021 REU hosted at Texas State University---Russell Barnes, Anton Dochtermann, Fran Herr, Cece Henderson and Ethan Partida---for useful discussions that led to this project.

\section*{Data Availability Statement}
This manuscript does not include any experimental data. All data generated or analyzed during this study are included in this published article.

\section*{Conflict of Interest}
On behalf of all authors, the corresponding author states that there is no conflict of interest.

\printbibliography

\end{document}